\documentclass[12pt,a4paper]{amsart}
\usepackage[margin=25mm]{geometry}
\usepackage[utf8]{inputenc}
\usepackage{amsfonts,amsmath,amssymb}
\def\<{\left<}
\def\>{\right>}

\def\ep{\varepsilon}

\def\R{\mathbb{R}}

\def\S{\mathbb{S}}
\def\H{\mathbb{H}}
\def\M{M_s^2}
\def\Mc{\mathbb{M}^3_c}
\def\tr{\text{\rm tr}}
\def\det{\text{\rm det}}
\def\cm{{\mathcal C}^\infty(M_s^2)}

\def\div{\text{\rm div}}

\newcommand{\ol}{\overline}

\newcommand{\va}{a}

\newcommand{\la}{\lambda}

\def\ka{\kappa}

\def\ColSep{\kern2pt}
\def\Sep#1{\kern#1pt}

\newtheorem{teor}{Theorem}
\newtheorem{lema}[teor]{Lemma}

\newtheorem{exa}{Example}
\newtheorem{cla}{Claim}

\newenvironment{ejem}{\begin{exa}\upshape}{\end{exa}}

\title[$L_1$-2-type surfaces in 3-dimensional De Sitter and anti De Sitter spaces]{$\pmb{L_1}$-2-type surfaces in 3-dimensional De Sitter and anti De Sitter spaces}
\author{S. Carolina Garc\'ia-Mart\'inez\and Pascual Lucas \and H. Fabi\'an Ram\'irez-Ospina}

\begin{document}
\maketitle

\begin{abstract}
Let $\M$ be an orientable surface immersed in the De Sitter space $\mathbb{S}_1^3\subset\R^4_1$ or anti de Sitter space $\mathbb{H}_1^3\subset\R^4_2$. In the case that $\M$ is of $L_1$-2-type
we prove that the following conditions are equivalent to each other: $\M$ has a constant principal curvature; $\M$ has constant mean curvature; $\M$ has constant second mean curvature. As a consequence, we also show that an $L_1$-2-type surface is either an open portion of a standard pseudo-Riemannian product, or a $B$-scroll over a null curve, or else its mean curvature, its Gaussian curvature and its principal curvatures are all non-constant.
\end{abstract}

\def\thefootnote{}
\footnotetext{%
 \hspace*{-\parindent}\emph{Mathematics Subject Classification 2020}: 53C50, 53B25, 53B30.\\
 \noindent\emph{Keywords:} Anti De Sitter space, De Sitter space, $L_1$-2-type surface, $B$-scroll, Cheng-Yau operator, Newton transformation.}

\section{Introduction}
\label{s:introduction}
The submanifolds of finite type in the Euclidean or pseudo-Euclidean space are submanifolds whose isometric immersion in the ambient space is constructed by using eigenfunctions of their Laplacian. This notion provides a very natural way to apply spectral geometry to the study of submanifolds (see \cite{C1984,CP}). This concept, originally defined for the Laplacian operator $\Delta$, can be generalized in a natural way to other operators such as the differential operator $\Box$ (sometimes denoted by $L_1$) introduced by Cheng and Yau in \cite{CY} for the study of surfaces with constant scalar curvature.

Let $\Mc$ denote the 3-dimensional De Sitter or anti De Sitter space, isometrically immersed in a pseudo-Euclidean space $\R^4_q$ of index $q$. An orientable surface $\psi:M^2_s\rightarrow\mathbb{M}_c^3\subset\R^4_q$ is said to be of $L_1$-2-type if the position vector $\psi$ admits the following spectral decomposition
\[
\psi=a+\psi_1+\psi_2,\quad L_1\psi_1=\lambda_1\psi_1, \quad L_1\psi_2=\lambda_2\psi_2,\quad \lambda_1\neq\lambda_2,\quad \lambda_i\in\mathbb{R},
\]
where $a$ is a constant vector in $\mathbb{R}_q^4$, and $\psi_1,\psi_2$ are $\R_q^4$-valued non-constant differentiable functions on $M_s^2$.

The second and third author in \cite{LR2017, LR2018} studied these kind of surfaces in the non-flat Riemannian space forms. It should be noted that the shape operator in the Riemannian case is always diagonalizable so that neither the techniques used nor the results obtained can be directly transferred to the Lorentzian case, since in this case the shape operator may be non-diagonalizable.

The main theorems of this paper are the following.

\noindent\textbf{Theorem A.}
\emph{Let $\psi:\M\to\Mc\subset\R^4_q$ be an orientable surface  of $L_1$-$2$-type. The following conditions are equivalent to each other:\vspace*{-\parskip}
\begin{enumerate}
\item[$1)$] $\M$ has a constant principal curvature.
\item[$2)$] $H$ is constant.
\item[$3)$] $H_2$ is constant.
\end{enumerate}}

As a consequence of this theorem, the known examples of $L_1$-2-type surfaces in $\Mc$ (see examples \ref{ej2}--\ref{ej4}) and the classification of isoparametric surfaces in $\Mc$ (see e.g. \cite{AFL94,X99,ZX}), we have the following characterization results of $L_1$-$2$-type surfaces in $\Mc$.

\noindent\textbf{Theorem B.}
\emph{Let $\psi: \M\to\S^3_1 \subset\R^4_1$ be an orientable surface of $L_1$-2-type. Then one of the following conditions is satisfied\vspace*{-\parskip}
\begin{itemize}\itemsep0pt
\item $\M$ an open portion of a standard pseudo-Riemannian product: \\ $\S^1_1(\sqrt{1-r^2})\times \S^1(r)$ or $\H^1(\sqrt{1-r^2})\times \S^1(r)$.
\item $M_1^2$ is a $B$-scroll over a null curve.
\item  $\M$ has non constant mean curvature, non constant Gaussian curvature, and non constant principal curvatures.
\end{itemize}}

\noindent\textbf{Theorem C.}
\emph{Let $\psi: M_s^2\to\H^3_1 \subset\R^4_2$ be an orientable surface of $L_1$-2-type. Then one of the following conditions is satisfied\vspace*{-\parskip}
\begin{itemize}\itemsep0pt
\item $\M$ is an open portion of a standard pseudo-Riemannian product:\\
    $\S^1_1(r)\times \H^1(-\sqrt{r^2+1)}$, or $\H^1(-r)\times\H^1(-\sqrt{1-r^2})$, or $\H^1_1(-r)\times \S^1(\sqrt{r^2-1})$
\item $M_1^2$ is a non-flat $B$-scroll over a null curve.
\item $\M$ has non constant mean curvature, non constant Gaussian curvature, and non constant principal curvatures.
\end{itemize}}

\section{Preliminaries}
\label{s:preliminaries}

In this section we recall some formulae and notions about surfaces in Lorentzian
space forms that will be used later on. Let $\R^{4}_q$ be the 4-dimensional pseudo-Euclidean space of index $q\in\{1,2,3\}$, with flat metric given by
\[
\<\cdot,\cdot\>=-\sum^q_{i=1}dx^2_i+\sum_{j=q+1}^4dx^2_j,
\]
where $x = (x_1, x_2, x_3, x_4)$ denotes the usual rectangular coordinates in $\R^4$. The De Sitter space of radius $r$ is defined by
\[
\S^{3}_1(r)= \big\{x\in\R^4_1\;:\; \<x, x\> = r^2\big\},
\]
and the anti De Sitter space of radius $-r$ is defined by
\[
\H^{3}_1(-r)= \big\{x\in\R^4_2\;:\; \<x, x\> = -r^2\big\}.
\]
It is well known that $\S^3_1(r)$ and $\H^3_1(-r)$ are Lorentzian totally umbilical hypersurfaces with constant sectional curvature $+1/r^2$ and $-1/r^2$, respectively.
In order to simplify our notation and computations, we will denote by $\Mc$
the De Sitter space $\S^3_1\equiv\S^3_1(1)$ or the anti De Sitter space $\H^3_1\equiv\H^3_1(-1)$ according to $c = 1$ or $c = -1$, respectively. We will use $\R^4_q$ to denote the corresponding pseudo-Euclidean space where $\Mc$ lives, so its metric is given by
\[
\<\cdot,\cdot\>= -dx^2_1 + c\;dx^2_2+dx^2_3+dx^2_4,
\]
and we can write
\[
\Mc= \big\{x\in\R^4_q \;:\; -x^2_1+cx^2_2+ x^2_3+x^2_4= c\big\}.
\]

Let $\psi:M_s^2\longrightarrow\Mc\subset\R^{4}_q$ be an isometric immersion of a surface $M_s^2$ into $\Mc$, and let $N$ be a unit vector field normal to $M_s^2$ in $\Mc$, where $\<N,N\>=\varepsilon=\pm1$. Here $s=0$ (resp. $s=1$) if the induced metric on the surface is Riemannian (resp. Lorentzian).  Let $\nabla^0$, $\ol\nabla$ and $\nabla$ denote the Levi-Civita connections on $\R^{4}_q$, $\Mc$ and $M_s^2$, respectively. Then the Gauss and Weingarten formulas are given by
\begin{equation}\label{F.G}
\nabla^0_XY=\nabla_XY+\varepsilon\< SX,Y\> N-c\< X,Y\> \psi,
\end{equation}
and
\[
SX=-\ol{\nabla}_XN=-\nabla^0_XN,
\]
for all tangent vector fields $X,Y\in\mathfrak{X}(M_s^2)$, where $S:\mathfrak{X}(M^2_s)\longrightarrow \mathfrak{X}(M_s^2)$ stands for the shape operator (or Weingarten endomorphism) of $M_s^2$, with respect to the chosen orientation $N$. If a vector field $X$ everywhere different from zero satisfies the condition $SX=\lambda X$, for a differentiable function $\lambda$, we will say that $X$ is a principal direction of $\M$ with associated principal curvature $\lambda$.

It is well-known (see, for instance, \cite[pp. 261--262]{ONeill}) that the shape operator $S$ of the surface $M^2_s$ can be expressed, in an appropriate frame, in one of the following types:
\begin{align}\label{Formas}
 \text{I. } S\approx
 \left[\begin{array}{@{\ColSep}c@{\Sep4}c@{\ColSep}}
    \kappa_1 &  0  \\
          0  & \kappa_2
 \end{array}\right];
 \kern1cm
 \text{II. } S\approx
 \left[\begin{array}{@{\ColSep}c@{\Sep4}r@{\ColSep}}
   \kappa &  -b   \\
   b      & \kappa
   \end{array}\right],\quad b\neq 0;
   \kern1cm
 \text{III. } S\approx
 \left[\begin{array}{@{\ColSep}c@{\Sep4}c@{\ColSep}}
   \kappa &  0    \\
   1      & \kappa
   \end{array}\right].
\end{align}
In cases I and II, $S$ is represented with respect to an orthonormal frame, whereas in case III the frame is pseudo-orthonormal. The characteristic polynomial $Q_S(t)$ of the shape operator $S$ is defined by $Q_S(t)={\rm det}(tI-S)=t^2+a_1t+a_2$, where the coefficients $a_i$ are given by
\begin{equation}\label{CPC1}
a_1=-\tr(S),\qquad a_2=\det(S).
\end{equation}
The mean curvature $H$ of $\M$ in $\Mc$ is given by $H=(\ep/2)\tr(S)$, and the coefficient $a_2$ is also called the \emph{second mean curvature or mean curvature of order $2$}, $H_2$.
The Newton transformation of $\M$ is the operator $P_1:\mathfrak{X}(\M)\longrightarrow \mathfrak{X}(\M)$ defined by
\begin{equation}
P_1=-2\ep HI+S,
\end{equation}
and by the Cayley-Hamilton theorem we have $S\circ P_1=-H_2I$.

The divergence of a vector field $X$ is the differentiable function defined as the trace of the operator $\nabla X$, where $\nabla X(Y):=\nabla_YX$. Analogously, the divergence of an operator $T:\mathfrak{X}(M^2_s)\longrightarrow\mathfrak{X}(M_s^2)$ is the vector field $\div(T)\in\mathfrak{X}(M_s^2)$ defined as the trace of $\nabla T$, where $\nabla T(X,Y)=(\nabla_{X}T)Y$.

We have the following properties of $P_1$ (see \cite{LR2011,LR2012}).

\begin{lema}\label{L1}
The Newton transformation $P_1$ satisfies:\vspace*{-.5\baselineskip}
\begin{list}{1}{\leftmargin5pt} 
\item[(a)] $P_1$ is self-adjoint and commutes with $S$.
\item[(b)] $\tr(P_1)=-2\varepsilon H$.
\item[(c)] $\tr(S\circ P_1)=-2H_{2}$.
\item[(d)] $\tr(S^2\circ P_1)=-2\varepsilon HH_2$.
\item[(e)] $\tr(\nabla_X S\circ P_1)=-\left\langle \nabla H_2,X\right\rangle$.
\item[(f)] $\div(P_1)=0$.
\end{list}\vspace*{-.75\baselineskip}
\end{lema}
Bearing this lemma in mind we obtain
\[
\div(P_1(\nabla f))=\tr\big(P_1\circ\nabla^2 f \big),
\]
where $\nabla^2f:\mathfrak{X}(M)\longrightarrow\mathfrak{X}(M)$ denotes the self-adjoint linear operator metrically equivalent to the Hessian of $f$, given by $\<\nabla^2f(X),Y\>=\<\nabla_X(\nabla f),Y\>$, for all $X,Y\in\mathfrak{X}(M_s^2).$
Associated to the Newton transformation $P_1$, we can define the second-order linear differential operator $L_1:\cm\longrightarrow\cm$ given by
\begin{align}\label{E9}
L_1(f)=\tr\big(P_1\circ\nabla^2 f \big).
\end{align}

An interesting property of $L_1$ is the following
\begin{align}
L_1(fg)=gL_1(f)+ fL_1(g) +2\<P_1(\nabla f),\nabla g\>,\label{Lkfg}
\end{align}
for every couple of differentiable functions $f,g\in C^\infty(M_s^2)$.

Now, we are going to compute $L_1$ acting on the coordinate components of the immersion $\psi$, that is, a function given by $\<e,\psi\>$, where $e\in \mathbb{R}^{4}_q$ is an arbitrary fixed vector.

A direct computation shows that
\begin{equation}\label{E10-}
\nabla\<e,\psi\>=e^\top=e-\varepsilon\<e,N\>N-c\<e,\psi\>\psi,
\end{equation}
where $e^\top\in \mathfrak{X}(M_s^2)$ denotes the tangential component of $e$. Taking covariant derivative in~(\ref{E10-}), and using that $\nabla^0_Xe=0$, jointly with the Gauss and Weingarten formulae, we obtain
\begin{equation}\label{E11}
\nabla_X\nabla\<e,\psi\>=\varepsilon\<e,N\> SX-c\<e,\psi\> X,
\end{equation}
for every tangent vector field $X\in \mathfrak{X}(M_s^2)$. Finally, by using (\ref{E9}) and Lemma~\ref{L1}, we find that
\begin{align}
L_1\<e,\psi\>=\ep\<e,N\>\tr(S\circ P_1)-c\<e,\psi\>\tr(P_1)
               =-2\ep H_2\<e,N\>+2\varepsilon cH\<e,\psi\>,\label{E12}
\end{align}
and then we get
\begin{equation}\label{E13}
L_1\psi=-2\varepsilon H_{2}N+2\varepsilon c H\psi.
\end{equation}
We will now compute $L_1$ acting on a function given by $\<e,N\>$. A straightforward computation yields
\begin{equation}\label{nablaconN}
\nabla\<e,N\>=-Se^\top,
 \end{equation}
that jointly with the Weingarten formula leads to
\begin{align*}
\nabla_X\nabla\<e,N\>=-(\nabla_{e^\top}S)X-\varepsilon\<e,N\>S^2X+c\<e,\psi\>SX,
\end{align*}
for every tangent vector field $X$. This equation, jointly with Lemma~\ref{L1} and~(\ref{E9}), yields
\begin{equation}\label{LNa}
L_1\<e,N\>
    =\<\nabla H_{2},e^\top\>+2HH_2\<e,N\>-2c H_{2}\<e,\psi\>,
\end{equation}
and then
\begin{equation}\label{E16}
L_1N=\nabla H_{2}+2HH_2N-2cH_{2}\psi.
\end{equation}
On the other hand, equations~(\ref{Lkfg}) and (\ref{E12}) lead to
\begin{align*}
\ep L_1(L_1\<e,\psi\>)&=-2H_{2}L_1\<e,N\>-2L_1(H_{2})\<e,N\>-4\big\langle P_1(\nabla H_{2}),\nabla\<e,N\>\big\rangle\\
&\qquad +2cHL_1\<e,\psi\>+2cL_1(H)\<e,\psi\>+4c\big\langle P_1(\nabla H),\nabla\<e,\psi\>\big\rangle.
\end{align*}
From here, and using (\ref{E12}), \eqref{E10-} and~(\ref{LNa}), we get
\begin{align*}
\ep L_1\big(L_1\<e,\psi\>\big)&=-2H_{2}\big\langle \nabla H_{2},e\big\rangle+4\big\langle(S\circ P_1)(\nabla H_{2}),e\big\rangle+4c\big\langle P_1(\nabla H),e\big\rangle \\
&\quad +\big[-4\ep HH_{2}(c+\ep H_2)-2L_1(H_{2})\big]\big\langle e,N\big\rangle\\
&\quad +\big[4cH^2_{2}+4\ep H^2+2cL_1(H)\big]\< e,\psi\>.
\end{align*}
Therefore, we obtain
\begin{align}\label{J1}
\ep L_1^2\psi&=\big[4cP_1(\nabla H)-3\nabla H_2^2]+\big[-4\ep HH_2(c+\ep H_2)-2L_1(H_2)\big]N\nonumber\\
&\quad+\big[4cH_2^2+4\ep H^2+2cL_1(H)\big]\psi.
\end{align}

\subsection*{Equations characterizing the $\pmb{L_1}$-2-type surfaces}

Let us suppose that $M_s^2$ is a $L_1$-2-type surface in $\R^4_q$, that is, the position vector field $\psi$ of $M^2_s$ in $\R^4_q$ can  be written as follows
\[
\psi=a+\psi_1+\psi_2,\quad L_1\psi_1=\lambda_1\psi_1, \quad L_1\psi_2=\lambda_2\psi_2,\quad \lambda_1\neq\lambda_2,\ \lambda_i\in\mathbb{R},
\]
where $a$ is a constant vector in $\mathbb{R}_q^4$, and $\psi_1,\psi_2$ are $\R_q^4$-valued non-constant differentiable functions on $M^2_s$.
Since $L_1\psi=\lambda_1\psi_1+\lambda_2\psi_2$ and $L_1^2\psi=\lambda_1^2\psi_1+\lambda_2^2\psi_2$, an easy computation shows that
\begin{align}\label{Es20}
L_1^2\psi=(\lambda_1+\lambda_2)L_1\psi-\lambda_1\lambda_2(\psi-\va).
\end{align}
From \eqref{E13} and \eqref{Es20} we obtain
\begin{align*}
L_1^2\psi=&\lambda_1\lambda_2\va^\top+[-2\ep(\lambda_1+\lambda_2)H_2+\ep\lambda_1\lambda_2\left\langle \va,N\right\rangle]N\\
&+[2\ep c(\lambda_1+\lambda_2)H-\lambda_1\lambda_2+c\lambda_1\lambda_2\left\langle \va,\psi\right\rangle]\psi.
\end{align*}
This equation, jointly with (\ref{J1}), yields the following equations, that characterize the $L_1$-2-type surfaces in $\Mc$:
\begin{align}
\lambda_1\lambda_2\va^\top&=-3\ep \nabla H_2^2+4\varepsilon  cP_1(\nabla H),\label{tan} \\
\lambda_1\lambda_2 \left\langle \va,N\right\rangle&=2(\lambda_1+\lambda_2)H_2-4\ep HH_2(c+\ep H_2)-2L_1(H_2),\label{nor1}\\
\lambda_1\lambda_2\left\langle \va,\psi\right\rangle&=4\ep H_2^2+4cH^2-2\ep(\lambda_1+\lambda_2)H+c\lambda_1\lambda_2+2\varepsilon L_1 (H).\label{nor2}
\end{align}

\section{Some Examples}
In this section we will show examples not only of $L_1$-2-type surfaces into $\mathbb{M}_c^3$ but also some surfaces that will be useful later in order to give the classification results.

\begin{ejem}(\emph{Totally umbilical surfaces in $\Mc$})\quad\label{ej2}
Let us begin by showing that the totally umbilical surfaces in $\Mc$ are of $L_1$-1-type.
As is well known, the totally umbilical surfaces in $\Mc$ are obtained as the intersection of $\Mc$ with a hyperplane of $\mathbb{R}^{4}_q$, and the causal character of the hyperplane determines the type of the surface. More precisely, let $\va\in \mathbb{R}^{4}_q$ be a non-zero constant vector with $\<\va,\va\>\in\{-1,0,1\}$, and we take the differentiable function $f_{\va}:\Mc\rightarrow\mathbb{R}$ defined by $f_{\va}(x)=\<a,x\>$. It is not difficult to see that for every $\tau\in \mathbb{R}$, with $\<\va,\va\>-c\tau^2\neq 0$, the set
\(
M(\tau)=f^{-1}_\va(\tau)=\{x\in\Mc\;|\;\<\va,x\>=\tau\}
\)
is a totally umbilical surface in $\Mc$. Its Gauss map $N$ and shape operator $S$ are given by
$N(x)=(1/\delta)(a-c\tau x)$ and $SX=(c\tau/\delta)X$,
where $\delta=\sqrt{|\<\va,\va\>-c\tau^2|}$. Now, it is easy to obtain
$H=\ep c\tau/\delta$ and $H_2=\tau^2/\delta^2$,
where $\ep=\<N,N\>$. From here we get that $M(\tau)$ has constant Gaussian curvature
\[
K=c+\ep H_2=\frac{c\<a,a\>}{\< \va,\va\>-c\tau^2},
\]
and ${M(\tau)}$ is a Riemannian or Lorentzian surface according to $\<\va,\va\>-c\tau^2$ is negative or positive, respectively. Now we will analyze the distinct possibilities.

\noindent\textbullet\kern8pt Case $c=1$. Then ${M(\tau)}\subset\Mc=\mathbb{S}_1^{3}\subset\mathbb{R}_1^{4}$ and we have:\vspace*{-\topskip}
\begin{enumerate}
\item[i)] If $\<\va,\va\>=-1$, then $K=1/(\tau^2+1)$, $\ep=-1$, and ${M(\tau)}$ is isometric to a round sphere of radius $\sqrt{\tau^2+1}$, ${M(\tau)}\equiv\mathbb{S}^2(\sqrt{\tau^2+1})$.
\item[ii)] If $\<\va,\va\>=0$, then $\tau\neq 0$, $K=0$, $\ep=-1$, and ${M(\tau)}$ is isometric to the Euclidean plane, ${M(\tau)}\equiv\mathbb{R}^2$.
Then ${M(\tau)}$ is a flat totally umbilic surface.
\item[iii)] If $\<\va,\va\>=1$, then either $|\tau|>1$, $K=-1/(\tau^2-1)$, $\ep=-1$, and ${M(\tau)}$ is isometric to the hyperbolic plane of radius $-\sqrt{\tau^2-1}$, ${M(\tau)}\equiv\mathbb{H}^2(-\sqrt{\tau^2-1})$, or $|\tau|<1$, $K=1/(1-\tau^2)$, $\ep=1$, and ${M(\tau)}$ is isometric to the De Sitter space of radius $\sqrt{1-\tau^2}$, ${M(\tau)}\equiv\mathbb{S}^2_1(\sqrt{1-\tau^2})$.
\end{enumerate}
\noindent\textbullet\kern8pt Case $c=-1$. Then ${M(\tau)}\subset\Mc=\mathbb{H}_1^{3}\subset\mathbb{R}_2^{4}$ and we have:\vspace*{-\topskip}
\begin{enumerate}
\item[i)] If $\<\va,\va\>=-1$, then either $|\tau|>1$, $K=1/(\tau^2-1)$, $\ep=1$, and ${M(\tau)}$ is isometric to a De Sitter space of radius $\sqrt{\tau^2-1}$, ${M(\tau)}\equiv\mathbb{S}^2_1(\sqrt{\tau^2-1})$, or $|\tau|<1$, $K=-1/(1-\tau^2)$, $\ep=-1$, and $M_\tau$ is isometric to a hyperbolic plane of radius $-\sqrt{1-\tau^2}$, ${M(\tau)}\equiv\mathbb{H}^2(-\sqrt{1-\tau^2})$.
\item[ii)] If $\<\va,\va\>=0$, then $\tau\neq0$, $K=0$, $\ep=1$, and ${M(\tau)}$ is isometric to the Lorentz-Minkowski plane, ${M(\tau)}\equiv\mathbb{R}^2_1$. Then ${M(\tau)}$ is a flat totally umbilical surface.
\item[iii)] If $\<\va,\va\>=1$, then $K=-1/(\tau^2+1)$, $\ep=1$, and ${M(\tau)}$ is isometric to the Lorentzian hyperbolic plane, ${M(\tau)}\equiv\mathbb{H}^2_1(-\sqrt{\tau^2+1})$.
\end{enumerate}
Bearing (\ref{E13}) in mind, we obtain
\[
L_1\psi=\lambda\psi+b,\quad\text{where }\lambda=\frac{2\tau}{\delta^3}\big(\ep c\tau^2+\delta^2\big)\textrm{ and }b=-\frac{2\ep\tau^2}{\delta^3}\va.
\]
We distinguish three cases: \vspace{-0.2cm}
\begin{itemize}
\item If $\tau=0$ (and so $S=0$), then $L_1\psi=0$ and $M(0)$ is a null $L_1$-$1$-type surface (for $c=1$, $M(0)\equiv\S^2$ or $M(0)\equiv \S^2_1$, and for $c=-1$, $M(0)\equiv\H^2$ or $M(0)\equiv\H^2_1$).
\item If $\ep c\tau^2+\delta^2=0$, then $L_1\psi=b\neq0$ and so $M(\tau)$ is of infinite $L_1$-type. Observe that if $c=1$ then ${M(\tau)}\equiv\R^2$, and if $c=-1$ then  ${M(\tau)}\equiv\R_1^2$. In any case, ${M(\tau)}$ is a flat totally umbilical surface.
\item If $\la\neq0$, then we can write $\psi=\psi_0+\psi_1$ with $\psi_0=-(1/\lambda)b$ and  $\psi_1=\psi+(1/\lambda)b$. Then $L_1\psi_1=\la\psi_1$ showing that ${M(\tau)}$ is of $L_1$-$1$-type.
\end{itemize}
\end{ejem}

\begin{ejem}(\textit{Standard pseudo-Riemannian products in $\Mc$})\quad\label{ej3}
In this example we show, as in the Riemannian case, that the pseudo-Riemannian products are $L_1$-2-type surfaces.
Let $f:\Mc\longrightarrow\R$ be the differentiable function defined by
\[
f(x)=cx^2_1-\delta_2x_2^2+\delta_3 x_3^2+\delta_4x^2_{4},
\]
where $\delta_2,\delta_3,\delta_4\in\{0,1\}$ with $\delta_2+\delta_3+\delta_4=1$.
In short, $f(x)=\<Dx,x\>$, where $D$ is the diagonal matrix $D=\textrm{diag}[c,\delta_2,\delta_3,\delta_4]$. Then, for every $r>0$, and $\rho=\pm1$ with $r^2-c\rho\neq0$, the level set $M^2=f^{-1}(\rho r^2)$ is a surface in $\Mc$, unless it is empty.

The Gauss map and the shape operator are given by
\begin{equation}\label{Nex3}
N(x)=\frac{1}{r\sqrt{\big|\rho-cr^2\big|}}\ (Dx-\rho c r^2x)\quad\text{and}\quad
S=\frac{-1}{r\sqrt{\big|\rho-cr^2\big|}}\begin{bmatrix}
	1-\rho c r^2\\
	&-\rho c r^2
\end{bmatrix}.
\end{equation}
Therefore, by using~(\ref{E13}) we get that
$L_1\psi=\lambda_1D\psi+\lambda_2(\psi-D\psi)$,
where
\[
\lambda_1=\frac{2\ep H_{2}(\rho cr^2-1)}{r\sqrt{\big|\rho-cr^2\big|}}+2\ep c H\kern1cm \textrm{and} \kern1cm \lambda_2=\frac{2\ep H_{2}\rho c r^2}{r\sqrt{\big|\rho-cr^2\big|}}+2\ep c H.
\]
If we put $\psi_1=D\psi$ and $\psi_2=\psi-D\psi$, then $\psi=\psi_1+\psi_2$,
$L_1\psi_1=\lambda_1\psi_1$, and $L_1\psi_2=\lambda_2\psi_2$. Therefore, $M^2$ is an $L_1$-$2$-type surface in $\R^4_q$.

The following two tables show the distinct pseudo Riemannian products in $\Mc$.

\begin{center}
\def\arraystretch{1.5}
\begin{tabular}{|c|c|c|r|p{82mm}|}
\hline
\multicolumn{5}{|c|}{Standard products in $\mathbb{S}_1^{3}$}\\\hline
$\delta_2$ & $\delta_3$ & $\delta_4$ & $\rho$ & surface \\\hline
1 & 0 & 0 & $-1$ & $\mathbb{H}^1(-r)\times\mathbb{S}^{1}(\sqrt{1+r^2})$ \\\hline
1 & 0 & 0 & $1$ & $\mathbb{S}^1_1(r)\times\mathbb{S}^{1}(\sqrt{1-r^2})$\\\hline
0 & 1 & 0 & 1 &
\begin{picture}(0,0)
\put(0,-10){$\mathbb{S}^1(r)\times\mathbb{S}^{1}_1(\sqrt{1-r^2})$ or $\mathbb{S}^1(r)\times\mathbb{H}^{1}(-\sqrt{r^2-1})$}
\end{picture} \\
0 & 0 & 1 & 1 &
 \\\hline
\end{tabular}

\begin{tabular}{|c|c|c|r|p{90mm}|}
\hline
\multicolumn{5}{|c|}{Standard products in $\mathbb{H}_1^{3}$}\\\hline
$\delta_2$&$\delta_3$&$\delta_4$&$\rho$&surface \\\hline
1 & 0 & 0 & $-1$ & $\mathbb{H}^1_1(-r)\times\mathbb{S}^{1}(\sqrt{r^2-1})$\\\hline
0 & 1 & 0 & 1 &
\begin{picture}(0,0)
\put(0,-10){$\mathbb{S}^1_1(r)\times\mathbb{H}^{1}(-\sqrt{1+r^2})$}
\end{picture}\\
0 & 0 & 1 & 1 & \\\hline
0 & 1 & 0 & $-1$ &
\begin{picture}(0,0)
\put(0,-10){$\mathbb{H}^1(-r)\times\mathbb{S}^{1}(\sqrt{r^2-1})$ or $\mathbb{H}^1(-r)\times\mathbb{H}^{1}(-\sqrt{1-r^2})$}
\end{picture} \\
0 & 0 & 1 & $-1$ &  \\\hline
\end{tabular}
\end{center}
\end{ejem}

\begin{ejem}(\textit{Complex circle})\label{ej4a}\quad
This example shows a surface satisfying an equation similar to \eqref{Es20} but it is not a surface of $L_1$-2-type.
Given a complex number $k=a+bi$, where $a,b\in\mathbb{R}$ such that $a^2-b^2=-1$ and $ab\neq0$, the complex circle of radius $k$ is defined by Magid in \cite{Magid2} as follows. Since $\mathbb{C}^{2}$ can be identified with $\R_{2}^{4}$ by sending $(z,w)=(u_1+iu_2,x+iy)$ to $(u_1,x,u_2,y)$, the mapping $\psi: \mathbb{C}\equiv\mathbb{R}_1^2\rightarrow\mathbb{H}_1^3\subset\mathbb{C}^2\equiv\mathbb{R}_2^4$ given by $\psi(z)=k(\cos z, \sin z)$
is an isometric immersion of $\R_{1}^{2}$ into $\R_{2}^{4}$ with parallel second fundamental form. $\psi$ parameterizes a Lorentzian surface in $\H^3_1$ known as the \emph{complex circle} of radius $k=a+bi$, and the corresponding shape operator $S$ is expressed as
\[
\left[\begin{array}{@{\ColSep}c@{\Sep4}c@{\ColSep}}
\alpha & -\beta \\[4pt]
\beta  & \alpha
\end{array}
\right],
\quad\text{with }\alpha=\dfrac{2ab}{a^2+b^2}\quad\text{and}\quad\beta=\dfrac{1}{a^2+b^2}.
\]
Thus, a complex circle is a flat surface in $\H^3_1$ with constant mean curvature $H=\alpha$.
Finally, a straightforward computation yields
\[
L^2_1\psi=\frac{-4}{(a^2+b^2)^2}\psi,
\]
showing that the complex circle satisfies an equation similar to \eqref{Es20}. However, it is not an $L_1$-2-type surface since there are no two distinct real numbers $\lambda_1,\lambda_2$ satisfying \eqref{Es20}.
\end{ejem}

\begin{ejem}(\textit{$B$-scroll})\label{ej4}\quad
Finally, we will show an example with non-diagonalizable shape operator. This example also shows the importance of the Gaussian curvature in concluding whether or not the surface is of $L_1$-2-type.
Let $\gamma(s)$, $s\in I$, be a null curve in $\Mc\subset\mathbb{R}_q^4$ with associated Cartan frame $\{A=\gamma',B,C\}$, i.e., $\{A,B,C\}$ is a pseudo-orthonormal frame of vector fields along $\gamma(s)$
\[
\left\langle A,A\right\rangle=\left\langle B,B\right\rangle=0, \left\langle A,B\right\rangle=-1,
\left\langle A,C\right\rangle=\left\langle B,C\right\rangle=0, \left\langle C,C\right\rangle=1,
\]
such that
\[
\gamma'(s)=A(s),\quad
C'(s)=-a_0A(s)-\kappa(s)B(s),
\]
where $a_0$ is a nonzero constant and $\kappa(s)\neq0$ for all $s$. Then the map $\psi:I\times \R\to\Mc\subset\R^4_q$, given by $\psi(s,u)=\gamma(s)+uB(s)$, defines a Lorentzian surface $M_1^2$ in $\Mc$ known as a $B$-scroll on the null curve $\gamma$ (see \cite{DN} and \cite{Graves}).\\[\parskip]
A simple calculation shows that the vector field $N(s,u)$ given by $N(s,u)=-a_0uB(s)+C(s)$ defines a unit normal vector field to $M^2_1$ in $\Mc$, and the shape operator $S$ is expressed in the usual tangent frame $\big\{\partial_s \psi,\partial_u \psi\big\}$ as
\[
S=\left[\begin{array}{@{\ColSep}c@{\Sep4}c@{\ColSep}}
a_0 & 0 \\[4pt]
\kappa(s) & a_0
\end{array}\right],
\]
whose minimal polynomial is $(t-a_0)^2$. Consequently, a $B$-scroll is a Lorentzian surface with constant mean curvature $H=a_0$ and constant Gauss curvature $K=c+a_0^2$.

Now, let us check if the $B$-scroll is of $L_1$-$2$-type. From (\ref{J1}) and \eqref{E13} we find
\begin{align*}
L^2_1\psi&=(2HK)L_1\psi.
\end{align*}
Therefore, a non-flat $B$-scroll is a null $L_1$-$2$-type surface, whereas a flat $B$-scroll (i.e, $c=-1$, $H_2=a_0^2=1$) is an $L_1$-biharmonic surface of infinite type.
\end{ejem}

\section{Main results}
\label{s:main.results}

Now, we are in position to present our first main result.
\begin{teor}\label{SC1}
Let $\psi:M_s^2\to\Mc\subset\R^4_q$ be an orientable surface of $L_1$-$2$-type. Then $H$ is constant if and only if $H_2$ is constant.
\end{teor}
\begin{proof}
Let us suppose that $H$ is constant. Our goal is
to prove that $H_2$ is also constant. Otherwise, let us consider the
non-empty open set
\[
\mathcal{U}_2=\big\{p\in M_s^2\;:\;\nabla H^2_2 (p)\neq0\big\}.
\]
By taking covariant derivative in (\ref{nor2}) we have $\la_1\la_2a^\top = 4\ep\nabla H^2_2$. From here and (\ref{tan}) we
get $\nabla H_2^2 = 0$ on $\mathcal{U}_2$, which is a contradiction. Hence, $H_2$ is constant.

Conversely, let us suppose now that $H_2$ is constant, and
consider the open set
\[
\mathcal{U}_1 = \big\{p\in M^2_s\;:\;\nabla H^2(p)\neq0\big\}.
\]
Our goal is to show that $\mathcal{U}_1$ is empty. Otherwise, by taking covariant derivative in (\ref{nor1})  we get
\[
\lambda_1\lambda_2Sa^\top=4\ep H_2(c+\ep H_2)\nabla H.
\]
From (\ref{tan}), and bearing in mind that $S\circ P_1 =-H_2I$, we get $\lambda_1\lambda_2Sa^\top= -4\ep c H_2\nabla H$, so that
\[
4\ep H_2(2c+\ep H_2)\nabla H=0.
\]
Consequently, on $\mathcal{U}_1$ we have either $H_2\equiv -2\ep c$ or $H_2\equiv 0$. We will study each case separately.

\noindent\textbf{Case 1:} $H_2\equiv -2\ep c$. By applying the operator $L_1$ on both sides of (\ref{nor1}), and using (\ref{nor2}), we get
\[
\lambda_1\lambda_2L_1\<a,N\>=-8L_1(H)=-4\big[\ep \lambda_1\lambda_2\<a,\psi\>-4\ep c H^2+2(\lambda_1+\lambda_2)H-\ep c \lambda_1\lambda_2-16\big].
\]
By using (\ref{LNa}) we obtain
\[
-c\lambda_1\lambda_2\<a,N\>H+\lambda_1\lambda_2\<a,\psi\>=-\lambda_1\lambda_2\<a,\psi\>+4cH^2- 2\ep(\lambda_1+\lambda_2)H+c\lambda_1\lambda_2+16\ep,
\]
and from (\ref{nor1}) we find
\[
\lambda_1\lambda_2\<\va,\psi\>=-2cH^2-3\ep(\lambda_1+\lambda_2)H+\frac{c\lambda_1\lambda_2}{2}+8\ep.
\]
By taking covariant derivative here, and using (\ref{tan}), we get
\[
\big[-4cH-3\ep(\lambda_1+\lambda_2)\big]\nabla H=-8cH\nabla H+4\ep c S(\nabla H),
\]
Now, by applying $S$ on both sides of this equation, and taking in mind that $S\circ P_1=2\ep c I$, we have
\[
\big[-4cH-3\ep(\lambda_1+\lambda_2)\big]S(\nabla H)=4\ep c S\circ P_1(\nabla H)=8\nabla H.
\]
From the last two equations we deduce
\[
16H^2-9(\lambda_1+\lambda_2)^2+32\ep c=0,
\]
and then $H$ is constant on $\mathcal{U}_1$, which is a contradiction.

\noindent\textbf{Case 2:} $H_2\equiv 0$, then $S(S-2\ep HI)=0$. Since a totally umbilical surface is not of $L_1$-2-type, we can assume that $M_s^2$ has two distinct principal curvatures $\kappa_1=0$ and $\kappa_2=2\ep H\neq0$.
Let $\{E_1,E_2\}$, with $\<E_i,E_i\>=\ep_i$, be a local orthonormal frame of principal directions of $S$ such that
\[
S=
\begin{pmatrix}
0&0\\
0&2\ep H
\end{pmatrix}.
\]
Write $\nabla H=\alpha E_1+\beta E_2$, where $\alpha=\ep_1 E_1(H)$, then we have $S(\nabla H)=2\ep H\beta E_2$. On the other hand, the Codazzi equation $(\nabla_{E_1}S)E_2=(\nabla_{E_2}S)E_1$ yields
\[
\ep_2 E_1(H)=-H\<\nabla_{E_2}E_1,E_2\>\quad\text{and}\quad 0=\<\nabla_{E_1}E_2,E_1\>.
\]
From here we deduce
$$\begin{array}{lllll}
	
\end{array}$$
\[
\begin{array}{l@{\quad}l@{\quad}l}
\nabla_{E_1}E_1=0,& \nabla_{E_1}E_2=0,\\[0.3cm]
\nabla_{E_2}E_1=-\dfrac{\ep_1\alpha}{H}E_2,&  \nabla_{E_2}E_2=\dfrac{\ep_1\ep_2\alpha}{H}E_1, &[E_1,E_2]=\dfrac{\ep_1\alpha}{H}E_2,
\end{array}
\]
and then the curvature tensor is given by (\cite{ONeill})
\[
R(E_1,E_2)E_1=\nabla_{E_1}\nabla_{E_2}E_1-\nabla_{E_2}\nabla_{E_1}E_1-\nabla_{[E_1,E_2]}E_1
=\Big[-E_1\Big(\frac{\ep_1\alpha}{H}\Big)+\Big(\frac{\ep_1\alpha}{H}\Big)^2\Big]E_2.
\]
On the other hand, since $\Mc$ is of constant curvature we have
\[
R(E_1,E_2)E_1=-\ep_1 cE_2.
\]
From the last two equations we have
\begin{align}\label{o1}
HE_1(\alpha)=cH^2+2\ep_1\alpha^2.
\end{align}
Now, since $P_1(E_1)=-\kappa_2E_1$ and $P_1(E_2)=0$, we get
\begin{align}\label{o2}
L_1(H)=-2\ep HE_1(\alpha),
\end{align}
that jointly with (\ref{nor2}) yields
\[
\lambda_1\lambda_2\<a,\psi\>=-2\ep(\lambda_1+\lambda_2)H+c\lambda_1\lambda_2-8\ep_1\alpha^2,
\]
and from here we get
\begin{align}\label{o3}
E_1(\lambda_1\lambda_2\<a,\psi\>)&=-2\ep(\lambda_1+\lambda_2)\ep_1\alpha-16\ep_1\alpha E_1(\alpha).	
\end{align}
However, from (\ref{tan}) we have
\begin{align}\label{d}
\lambda_1\lambda_2a^\top=-8 cH\alpha E_1,
\end{align}
so that
\begin{align}
E_1(\lambda_1\lambda_2\<a,\psi\>)&=\<\lambda_1\lambda_2a^\top,E_1\>=-8\ep_1 cH\alpha,
\end{align}
that jointly with (\ref{o3}) implies
\begin{align}\label{o4}
-4E_1(\alpha)&=-2cH+\frac{\ep(\lambda_1+\lambda_2)}{2}.
\end{align}
From here, and using (\ref{o2}) and (\ref{nor2}), we find
\begin{align}\label{d0} \lambda_1\lambda_2\<\va,\psi\>=2cH^2-\frac{3\ep}{2}(\lambda_1+\lambda_2)H+c\lambda_1\lambda_2.
\end{align}
Taking gradient in (\ref{d0}) and using (\ref{tan}) we get
\begin{align}\label{m2}
\Big[4cH-\frac{3\ep}{2}(\lambda_1+\lambda_2)\Big]\nabla H&=4\ep c P_1(\nabla H)=-8cH\nabla H+4\ep c S(\nabla H).
\end{align}
Now, by applying $S$ on both sides of (\ref{m2}) we have
\[
\big[4cH-\frac{3\ep}{2}(\lambda_1+\lambda_2)\big]S(\nabla H)=0.
\]
From the last two equations we find
\[
\big[4cH-\frac{3\ep}{2}(\lambda_1+\lambda_2)\big]\Big[3\ep H-\frac{3c}{8}(\la_1+\la_2)\Big]\nabla H=0,
\]
and this implies that $H$ is constant on $\mathcal{U}_1$, which is a contradiction.  This finishes the proof of Theorem \ref{SC1}.
\end{proof}

\begin{teor}\label{SC2}
Let $\psi:\M\to\Mc\subset\R^4_q$ be an orientable surface  of $L_1$-$2$-type. The following conditions are equivalent to each other:\vspace{-\topsep}
\begin{enumerate}
\item[$1)$] $\M$ has a constant principal curvature.
\item[$2)$] $H$ is constant.
\item[$3)$] $H_2$ is constant.
\end{enumerate}
\end{teor}
\begin{proof}
From Theorem \ref{SC1} we know that conditions 2) and 3) are equivalent to each other. Therefore, it is sufficient to prove the equivalence of conditions 1) and 2).

If condition 2) holds then $\M$ is an isoparametric surface since condition 3) also holds. Then the shape operator $S$ of $\M$ in $\Mc$ is of type I (with constant $\kappa_1$ and $\kappa_2$) or type III (with constant $\kappa$), so we have condition 1). Observe that $S$ cannot be of type II since this kind of surfaces do not exist in $\S^3_1$ and the complex circle in $\H^3_1$ is not of $L_1$-2-type.

Let us now assume that condition 1) is satisfied, then $S$ is of type I or type III. If $S$ is of type III then $\ka$ is constant, and therefore $H$ and $H_2$ are also constants.

Let us consider the case where $S$ is diagonalizable with two distinct principal curvatures $\ka_1\neq\ka_2$, and assume without loss of generality that $\ka_1$ is a nonzero constant
(otherwise, $H_2=0$ and the Theorem \ref{SC1} applies). Consider the open set
\[
\mathcal{U} = \big\{p\in M^2_s\;:\;\nabla\ka_2^2(p) \neq 0\big\}.
\]
Our goal is to show that $\mathcal{U}$ is empty.
Otherwise, the equations (\ref{tan})--(\ref{nor2}) of a $L_1$-$2$-type surface can be rewritten in terms of $\ka_2$ as
follows
\begin{align}
\la_1\la_2a^\top &= [-6\ep\ka^2
_1\ka_2 -2c(\ka_1 + \ka_2)]\nabla\ka_2 + 2cS(\nabla\ka_2),\label{28}\\
\la_1\la_2\<a,N\> &=
2\ka_1\ka_2[\la_1 + \la_2-(\ka_1 + \ka_2)(c+\ep\ka_1\ka_2)]-2\ka_1L_1\ka_2,\label{29}\\
\la_1\la_2\<\va,\psi\> &= 4\ep\ka^2_1\ka^2_
2 + c(\ka_1 + \ka_2)^2 - (\la_1 + \la_2)(\ka_1 + \ka_2) + c\la_1\la_2+ L_1\ka_2.\label{30}
\end{align}
From (\ref{29}) and (\ref{30}) we find
\[
\la_1\la_2\<a, N\> =-2\ka_1\la_1\la_2 \<\va,\psi\> + 2\ka_1[
3\ep\ka^2_1\ka^2_2 + c\ka^2_1 + c\ka_1\ka_2 -(\la_1 + \la_2)\ka_1 + c\la_1\la_2-\ep\ka_1\ka^3_2].
\]
By applying the gradient here we obtain
\begin{equation}\label{31}
-\la_1\la_2Sa^\top = -2\ka_1\la_1\la_2a^\top + 2\ka^2_1[c +
6\ep\ka_1\ka_2-3\ep\ka_2^2]\nabla \ka_2.
\end{equation}
On the other hand, by (\ref{28}) we get
\begin{equation}\label{32}
\la_1\la_2Sa^\top = -6\ep\ka^2_1\ka_2S(\nabla\ka_2)-2c\ka_1\ka_2\nabla \ka_2
\end{equation}
Now, from (\ref{28}), (\ref{31}) and (\ref{32}) we deduce
\[
(3\ep\ka_1\ka_2 + 2c)S(\nabla\ka_2) = (12\ep\ka^2_1\ka_2 -3\ep\ka_1\ka^2_2 +  c\ka_2 + 3c\ka_1)\nabla\ka_2.
\]
Since $3\ka_1\ka_2 + 2c\neq 0$ (otherwise, $\ka_2$ would be constant), we deduce
\[
S(\nabla\ka_2) = f (\ka_1, \ka_2)\nabla\ka_2,\qquad f(\ka_1, \ka_2) =\frac{12\ep\ka^2_1\ka_2 -3\ep\ka_1\ka^2_2 +  c\ka_2 + 3c\ka_1}{
(3\ep\ka_1\ka_2 + 2c)}.
\]
This equation implies that either $f (\ka_1, \ka_2) = \ka_1$ or $f(\ka_1, \ka_2) = \ka_2$. In any case it follows
that $\ka_2$ is constant on $\mathcal{U}$, and this is a contradiction. This finishes the proof of Theorem \ref{SC2}.
\end{proof}

As a consequence of theorems \ref{SC1} and \ref{SC2}, the examples \ref{ej2}--\ref{ej4} and the classification of isoparametric surfaces in $\Mc$,  see \cite{AFL94,X99,ZX}, we have the following characterization of $L_1$-$2$-type surfaces in $\Mc$.

\begin{teor}
Let $\psi: \M\to\S^3_1 \subset\R^4_1$ be an orientable surface of $L_1$-2-type. Then
one of the following conditions is satisfied\vspace{-\topsep}
\begin{itemize}\itemsep0pt
\item $\M$ an open portion of a standard pseudo-Riemannian product: $\S^1_1(\sqrt{1-r^2})\times \S^1(r)$ or $\H^1(\sqrt{1-r^2})\times \S^1(r)$.
\item $M^2_1$ is a $B$-scroll over a null curve.
\item  $\M$ has non constant mean curvature, non constant Gaussian curvature, and non constant principal curvatures.
\end{itemize}
\end{teor}

\begin{teor}
Let $\psi: M_s^2\to\H^3_1 \subset\R^4_2$ be an orientable surface of $L_1$-2-type. Then one of the following conditions is satisfied\vspace{-\topsep}
\begin{itemize}\itemsep0pt
\item $\M$ is an open portion of a standard pseudo-Riemannian product: $\S^1_1(r)\times \H^1(-\sqrt{r^2+1)}$, or $\H^1(-r)\times\H^1(-\sqrt{1-r^2})$, or $\H^1_1(-r)\times \S^1(\sqrt{r^2-1})$
\item $M^2_1$ is a non-flat $B$-scroll over a null curve.
\item $\M$ has non constant mean curvature, non constant Gaussian curvature, and non constant principal curvatures.
\end{itemize}
\end{teor}

\section*{Acknowledgements}

This research is part of the grant PID2021-124157NB-I00, funded by MCIN/ AEI/ 10.13039/ 501100011033/ ``ERDF A way of making Europe''. Also supported by ``Ayudas a proyectos para el desarrollo de investigación científica y técnica por grupos compe\-ti\-ti\-vos'', included in the ``Programa Regional de Fomento de la Investigación Científica y Técnica (Plan de Actuación 2022)'' of the Fundación Séneca-Agencia de Ciencia y Tecnología de la Región de Murcia, Ref. 21899/PI/22.

\section*{Conflict of interest}
The authors declare that they have no conflict of interest.

\small\scshape
\parindent0pt
\parskip\baselineskip
\def\correo#1{\textup{#1}}

\noindent
S. Carolina Garc\'ia-Mart\'inez\\
Departamento de Matem\'aticas\\
Universidad Nacional de Colombia\\
Sede Bogot\'a, Colombia\\
\correo{sacgarciama@unal.edu.co}

\noindent
Pascual Lucas \\
Departamento de Matemáticas\\
Universidad de Murcia\\
Murcia, Espa\~na\\
\correo{plucas@um.es}

\noindent
H. Fabi\'an Ram\'irez-Ospina \\
Departamento de Matem\'aticas\\
Universidad Nacional de Colombia\\
Sede Bogot\'a, Colombia\\
\correo{hframirezo@unal.edu.co}

\end{document}